\documentclass[11pt]{amsart}

\usepackage{amssymb,amsmath,amsthm,newlfont}
\usepackage[shortlabels]{enumitem}

\theoremstyle{plain}
\newtheorem{thm}{Theorem}[section]

\newtheorem{cor}[thm]{Corollary}
\newtheorem{lem}[thm]{Lemma}

\theoremstyle{definition}
\newtheorem{defn}[thm]{Definition}
\newtheorem{example}[thm]{Example}

\theoremstyle{remark}

\begin{document}

\title[Asymptotically equicontinuous sequences of operators]{Asymptotically equicontinuous sequences of operators and a Banach--Steinhaus type theorem}

\author{Javad Mashreghi}
\address{D\'epartement de math\'ematiques et de statistique, Universit\'e Laval,
Qu\'ebec City (Qu\'ebec),  Canada G1V 0A6.}
\email{javad.mashreghi@mat.ulaval.ca}

\author{Thomas Ransford}
\address{D\'epartement de math\'ematiques et de statistique, Universit\'e Laval,
Qu\'ebec City (Qu\'ebec),  Canada G1V 0A6.}
\email[Corresponding author]{thomas.ransford@mat.ulaval.ca}

\thanks{Mashreghi supported by an NSERC Discovery Grant. Ransford  supported by grants from NSERC and the Canada Research Chairs program.}

\begin{abstract}  
We introduce the notion of an asymptotically equicontinuous sequence of linear operators, 
and use it to prove the following result. 
If $X,Y$ are topological vector spaces,  
if $T_n,T:X\to Y$ are continuous linear maps, 
and if $D$ is a dense subset of $X$, 
then the following statements are equivalent: 
(i) $T_nx\to Tx$ for all $x\in X$, and 
(ii) $T_n x\to Tx$ for all $x\in D$ and the sequence $(T_n)$ is asymptotically equicontinuous.
\end{abstract}

\maketitle

\section{\bf Introduction}\label{S:intro}

The following theorem is a standard exercise in  functional analysis.

\begin{thm}\label{T:norm}
Let $X$ and $Y$ be Banach spaces, let $(T_n)_{n\ge1}$ and $T$
be continuous linear maps from $X$ to $Y$, and let $D$ be a dense subset of~$X$. 
Then the following statements are equivalent:
\begin{enumerate}
\item[(i)] $T_nx\to Tx$ for all $x\in X$;
\item[(ii)] $T_nx\to Tx$ for all $x\in D$ and $\sup_n\|T_n\|<\infty$.
\end{enumerate}
\end{thm}

The implication (ii)$\Rightarrow$(i) is an easy $\epsilon/3$-argument.
The implication (i)$\Rightarrow$(ii) is an application of the 
Banach--Steinhaus theorem; in fact it dates back
to the original paper of Banach and Steinhaus \cite{BS27}.

The completeness of $Y$ is not really needed here,
since we can always embed $Y$ in its completion.
The completeness of $X$, however, is needed for
the Banach--Steinhaus theorem. 
Without it the implication (i)$\Rightarrow$(ii) may actually fail.

\begin{example}\label{X:incomplete}
Let $X=c_{00}$, the space of finitely supported
 sequences of real numbers, with the sup-norm, 
and let $Y=\mathbb{R}$. Let $\pi_n:c_{00}\to \mathbb{R}$ be the 
$n$-th coordinate functional,  let $T_n:=n\pi_n$ and let $T:=0$. 

For each $x\in c_{00}$, we have $T_nx=0$ for all sufficiently large $n$, so (i) holds. 
However, $\|T_n\|=n$ for all $n$, so (ii) fails.
\end{example}

Theorem~\ref{T:norm} has a corollary for weak topologies.

\begin{cor}\label{C:weak}
Let $X$ and $Y$ be Banach spaces, let $(T_n)_{n\ge1}$ and $T$
be continuous linear maps from $X$ to $Y$, and let $D$ be a weakly dense subset of $X$. 
Then the following statements are equivalent:
\begin{enumerate}
\item[(i)] $T_nx\to Tx$ weakly for all $x\in X$;
\item[(ii)] $T_nx\to Tx$ weakly for all $x\in D$ and $\sup_n\|T_n\|<\infty$.
\end{enumerate}
\end{cor}

\begin{proof}
The implication (i)$\Rightarrow$(ii) holds because, by the Banach--Steinhaus theorem, 
weakly bounded implies norm bounded. 

For  the implication (ii)$\Rightarrow$(i), let $\phi\in Y^*$, the dual of $Y$.
If (ii) holds, then $(\phi\circ T_n)x\to (\phi\circ T)x$  for all $x\in D$, 
hence also  for all $x$ in the linear span of $D$, 
and if the latter is weakly dense then it is also norm dense
(see e.g.\ \cite[Corollary to Theorem~3.12]{Ru91}). 
Thus  we may apply Theorem~\ref{T:norm} 
to deduce that $(\phi\circ T_n)x\to (\phi\circ  T)x$ for all $x\in X$.
As this holds for each $\phi\in Y^*$, we conclude that (i) holds.
\end{proof}

If $X$ and $Y$ happen to be dual spaces, then we may also ask
whether the analogue of Corollary~\ref{C:weak} holds for weak* topologies. 
This problem arose recently in \cite{GMR22} in the context of summability operators.
It turns out that the answer is negative. 
This time, interestingly, it is the implication (ii)$\Rightarrow$(i) that breaks down.

\begin{example}\label{X:weak*}
Let $X=\ell^\infty$, the space of bounded sequences, normed by the
sup-norm, and let $Y=\mathbb{R}$. 
Let $\pi_n:\ell^\infty\to\mathbb{R}$ be the $n$-th coordinate functional,  let $T_n:=\pi_n$ and let $T:=0$.
Let $D=c_0$, the subspace of $\ell^\infty$ consisting of sequences that tend to zero. 

Since the bidual of $c_0$ is $\ell^\infty$, it follows that $c_0$ is weak* dense in $\ell^\infty$
(see e.g.\ \cite[Chapter~4, Exercise~1]{Ru91}). 
Also $T_nx\to0$
for all $x\in c_0$ and $\sup_n\|T_n\|=1<\infty$, so (ii) holds.
However, if $x$ is the constant sequence $(1,1,\dots)$, then
$x\in\ell^\infty$ and $T_nx\not\to0$, so (i) fails.
\end{example}

One might reasonably argue that, to obtain a true weak* version of
Corollary~\ref{C:weak}, one should also replace the condition 
$\sup_n\|T_n\|<\infty$ by one that is more closely tied to the weak* topologies on $X$ and $Y$. 
A natural candidate is that the sequence $(T_n)$  be weak* equicontinuous, i.e., that for each weak* $0$-neighbourhood $V$ in $Y$, there exists a weak* $0$-neighbourhood $U$ in $X$ 
such that $T_n(U)\subset V$ for all $n$. With this change, it is  true that (ii) implies (i) (for the weak* topologies). However, as the following example shows, we then lose the implication (i)$\Rightarrow$(ii).

\begin{example}\label{X:equicts}
Let $X=\ell^2$, with the usual $\ell^2$-norm, and let $Y=\mathbb{R}$.
Let $\pi_n:\ell^2\to\mathbb{R}$ be the $n$-th coordinate functional, let 
$T_n:=\pi_n$ and let $T:=0$. 

For each $x\in \ell^2$, we have $T_nx\to0$ in $\mathbb{R}$, so (i) holds.
However, if $U$ is any weak* $0$-neighbourhood in $\ell^2$,
then $U$ contains a non-zero subspace of $\ell^2$
(see e.g.\ \cite[p.66]{Ru91}),
and it follows easily that $T_n(U)=\mathbb{R}$ for all $n$.
We conclude that  the sequence $(T_n)$ is not weak* equicontinuous, and so (ii) fails
in this setting.
\end{example}

In the article \cite{GMR22},  these difficulties were circumvented by exploiting
the structure of the particular operators involved.
But for general operators, the problem remains.
Our purpose in this article is to propose a solution,
by replacing the condition
$\sup_n\|T_n\|<\infty$ in Theorem~\ref{T:norm} 
with an appropriate condition
so that the equivalence (i)$\iff$(ii) holds for weak* topologies, 
and indeed for arbitrary topological vector spaces.


\section{Asymptotic equicontinuity}

Given a set $X$ and a sequence $(F_n)$ of subsets of $X$,
we write $\liminf_n F_n$  for the set of $x\in X$ that
belong to $F_n$ for all but finitely many $n$.

\begin{defn}
Let $X,Y$ be topological vector spaces,
and let $(T_n)$ be a sequence of continuous linear maps from $X$ to $Y$.
We say that $(T_n)$ is \emph{asymptotically equicontinuous} if,
for each $0$-neighbourhood $V$ in $Y$, the set 
$\liminf_nT_n^{-1}(V)$ is a $0$-neighbourhood in $X$.
\end{defn}

Let us spell this out explicitly: $(T_n)$ is asymptotically equicontinuous if,
for each $0$-neighbourhood $V$ in $Y$, 
there exists a $0$-neighbourhood $U$ in $X$ such that, whenever $x\in U$, then $T_nx\in V$ for all large enough $n$.

Clearly, if $(T_n)$ is equicontinuous with respect to $X$ and $Y$, 
then it is asymptotically equicontinuous. 

The converse is true if $X$ and $Y$ are Banach spaces.
Indeed, in this case, 
$(T_n)$ asymptotically equicontinuous implies that 
$\sup_{n}\|T_nx\|<\infty$ for each $x\in X$, 
which in turn implies that $\sup_{n}\|T_n\|<\infty$
by the Banach--Steinhaus theorem,
whence $(T_n)$ is equicontinuous.

However,  in general, asymptotically equicontinuous does not imply equicontinuous. 
For example, the sequence $(T_n)$ in Example~\ref{X:incomplete}, 
being unbounded in norm, is not equicontinuous. 
However it is asymptotically equicontinuous: this follows from Theorem~\ref{T:main} below, 
but it is also easy to verify directly.

We can now state our main result.

\begin{thm}\label{T:main}
Let $X$ and $Y$ be topological vector spaces,
let $(T_n)_{n\ge1}$ and $T$ be continuous linear maps from $X$ to $Y$,
and let $D$ be a dense subset of~$X$.
Then the following statements are equivalent:
\begin{enumerate}
\item[(i)] $T_n x\to Tx$ for all $x\in X$;
\item[(ii)] $T_n x\to Tx$ for all $x\in D$, and the sequence $(T_n)$
is asymptotically equicontinuous.
\end{enumerate}
\end{thm}

For the proof of Theorem~\ref{T:main}, we require a lemma.

\begin{lem}\label{L:main}
If $(S_n)$ and $(T_n)$ are asymptotically equicontinuous sequences
of linear maps from $X$ to $Y$, then so is $(S_n+T_n)$.
\end{lem}

\begin{proof}
Let $V$ be a $0$-neighbourhood in $Y$. 
Let $W$ be a $0$-neighbourhood in $Y$ such that $W+W\subset V$.
As $(S_n)$ and  $(T_n)$ are asymptotically equicontinuous sequences,
the sets  
$U_1:=\liminf S_n^{-1}(W)$ and $U_2:=\liminf T_n^{-1}(W)$ are $0$-neighbourhoods in $X$.
Set $U:=U_1\cap U_2$. 
Then $U$ is a $0$-neighbourhood in $X$ and, if $x\in U$, 
then, for all large enough $n$, we have $S_nx\in W$ and $T_nx\in W$, 
whence $(S_n+T_n)x\in W+W\subset V$. 
This shows that $U\subset \liminf(S_n+T_n)^{-1}(V)$.
\end{proof}

\begin{proof}[Proof of Theorem~\ref{T:main}]
(i)$\Rightarrow$(ii): Suppose that $T_nx\to Tx$ for all $x\in X$. 
Obviously this holds, in particular, for all $x\in D$.
Also, for each $0$-neighbourhood $V$ in $Y$, 
we have $\liminf (T_n-T)^{-1}(V)=X$, simply by the definition of convergence of $T_nx$ to $Tx$. 
Therefore the sequence $(T_n-T)$ is asymptotically equicontinuous. 
Obviously the constant sequence $(T)$ is asymptotically equicontinuous, 
so, by the lemma, $(T_n)$ is asymptotically equicontinuous.

\medskip
(ii)$\Rightarrow$(i): Suppose that the hypotheses in (ii) hold.
Set $R_n:=T_n-T$. Then $R_nx\to0$ for all $x\in D$, and by the lemma,
the sequence $(R_n)$ is asymptotically equicontinuous.
We need to show that $R_nx\to0$ for all $x\in X$.

Let $x\in X$ and   $V$ be a $0$-neighbourhood in $Y$. 
We shall prove  that $R_n x\in V$ for all large enough $n$.
We may choose another $0$-neighbourhood $W$ in $Y$ such that $W-W\subset V$. 
Since $(R_n)$ is asymptotically equicontinuous, $U:=\liminf R_n^{-1}(W)$ 
is a $0$-neighbourhood in~$X$.
Since $D$ is dense in $X$, there exists $x'\in D$ such that $x'\in x+U$. 
Since $x'-x\in U$, there exists $N$ such that
\[
n\ge N\quad\Rightarrow\quad R_n(x'-x)\in W.
\]
Also, since $x'\in D$, we have $R_nx'\to0$, so there exists $N'$ such that
\[
n\ge N'\quad\Rightarrow\quad R_n x'\in W.
\]
Hence, finally, 
\[
n\ge \max(N,N')\quad\Rightarrow\quad R_n x=R_n(x')-R_n(x'-x)\in W-W\subset V.
\]
This completes the proof.
\end{proof}


\section{Concluding remarks}

We have formulated the notion of asymptotic continuity 
for sequences of operators.
However, given that our main result, Theorem~\ref{T:main},
treats topological vector spaces that are not necessarily metrizable, 
it would perhaps be more logical to define asymptotic continuity for nets rather than sequences.
In this section, we discuss the (relatively minor) changes to the preceding section needed to achieve this.

Let $X$ be a set, and let $(F_\alpha)_{\alpha\in A}$ be a net
of subsets of $X$, i.e., a collection of subsets indexed by a directed set $A$. We write $\liminf_\alpha F_\alpha$ for the set of $x\in X$ with the following property: there exists $\alpha_0\in A$ (depending on $x$) such that
$x\in F_\alpha$ for all $\alpha\ge \alpha_0$.

\begin{defn}
Let $X,Y$ be topological vector spaces, and let $(T_\alpha)$
be a net of continuous linear maps from $X$ to $Y$. We say that 
$(T_\alpha)$ is \emph{asymptotically equicontinuous} if,
for each $0$-neighbourhood $V$ in $Y$, the set 
$\liminf_\alpha T_\alpha^{-1}(V)$ is a $0$-neighbourhood in $X$.
\end{defn}

The following results are the extensions of Lemma~\ref{L:main} and Theorem~\ref{T:main} to nets. The proofs are obtained by making the  obvious modifications to the arguments for sequences. We omit the details.

\begin{lem}
If $(S_\alpha)$ and $(T_\alpha)$ are asymptotically equicontinuous nets
of linear maps from $X$ to $Y$, indexed by the same directed set, then  $(S_\alpha+T_\alpha)$ is also an asymptotically equicontinuous net.
\end{lem}

\begin{thm}
Let $X$ and $Y$ be topological vector spaces,
let $(T_\alpha)$ be a net of  continuous linear maps from $X$ to $Y$,
let $T:X\to Y$ be another continuous linear map,
and let $D$ be a dense subset of~$X$.
Then the following statements are equivalent:
\begin{enumerate}
\item[(i)] $T_\alpha x\to Tx$ for all $x\in X$;
\item[(ii)] $T_\alpha x\to Tx$ for all $x\in D$, and the net $(T_\alpha)$
is asymptotically equicontinuous.
\end{enumerate}
\end{thm}

\end{document}